\documentclass[10pt]{amsart}
\usepackage{amsmath,amscd,amssymb,amsthm,amsbsy,amscd,stmaryrd}
\usepackage{mathrsfs}
\input{xy}

\usepackage{pstricks}
\usepackage{color}
\usepackage{graphicx}
\usepackage[all]{xy}
\usepackage {hyperref}

\newtheorem{theorem}{Theorem}[section]
\newtheorem{lemma}[theorem]{Lemma}

\theoremstyle{definition}
\newtheorem{definition}[theorem]{Definition}

\newtheorem{remark}[theorem]{Remark}

\numberwithin{equation}{section}

\begin{document}
\title[Donaldson's equation]{A priori estimates for Donaldson's equation over compact Hermitian manifolds}
\author{Yi Li}
\address{Department of Mathematics,
Shanghai Jiao Tong University, 800 Dongchuan Road, Shanghai, 200240 China}

\email{yilicms@gmail.com}




\begin{abstract} In this paper we prove a priori estimates for 
Donaldson's equation
\begin{equation*}
\omega\wedge(\chi+\sqrt{-1}\partial\bar{\partial}\varphi)^{n-1}
=e^{F}(\chi+\sqrt{-1}\partial\bar{\partial}\varphi)^{n},
\end{equation*}
over a compact complex manifold $X$ of complex dimension $n$, where $\omega$ and $\chi$ are arbitrary Hermitian metrics. Our estimates answer a question of Tosatti-Weinkove in \cite{TW1}.
\end{abstract}
\maketitle

\tableofcontents

\section{Introduction}\label{section1}

\subsection{Donaldson's equation over compact K\"ahler manifolds}

Let $(X,\omega)$ be a compact K\"ahler manifold of the complex dimension $n$, and $\chi$ another K\"ahler metric on $X$.
In \cite{D1}, Donaldson considered the following interesting equation
\begin{equation}
\omega\wedge\eta^{n-1}=c\eta^{n}, \ \ \ [\eta]=[\chi],\label{1.1}
\end{equation}
where $c$ is a constant, depending only on the K\"ahler classes of $[\chi]$ and $[\omega]$, given by
\begin{equation}
c=\frac{\int_{X}\omega\wedge\chi^{n-1}}{\int_{X}\chi^{n}}.\label{1.2}
\end{equation}
He noted that a necessary condition for equation (\ref{1.1}) is
\begin{equation}
nc\chi-\omega>0,\label{1.3}
\end{equation}
and then conjectured that the condition (\ref{1.3}) is also sufficient. For $n=2$, Chen \cite{C1} observed that in this case the equation (\ref{1.1}) reduces to a complex Monge-Amp\`ere equation completely solved by Yau on his celebrated work on Calabi's conjecture \cite{Yau1}.

\subsection{$J$-flow and Donaldson's equation}

To better understand the equation (\ref{1.1}), Donaldson \cite{D1} and Chen \cite{C1} independently discovered the $J$-flow whose critical point gives the equation (\ref{1.1}), and Chen showed that such flow always exists for all time. Using the $J$-flow, Chen \cite{C2} proved that if $n=2$ and the holomorphic bisectional curvature of $\omega$ is nonnegative then the $J$-flow converges to a critical metric. Later, the curvature assumption was removed by Weinkove \cite{W1} and hence gave an alternative proof of Donaldson's conjecture on K\"ahler surfaces. For higher dimensional case, Weinkove \cite{W2} solved Donaldson's conjecture on a slightly strong condition
\begin{equation}
nc\chi-(n-1)\omega>0\label{1.4}
\end{equation}
using the $J$-flow. For more detailed discussions and related works, we refer to \cite{FL1, FL2, FLM, FLSW, SW, SW1}.

\subsection{Donaldson's equation over compact Hermitian manifolds}

Recently, the complex Monge-Amp\`ere equation over compact Hermitian manifolds was solved Tosatti and Weinkove \cite{TW1, TW2}. Other interesting estimates
can be found in \cite{TW3, Z1, ZZ1}. A parabolic proof was late given by Gill \cite{Gill1} by considering a parabolic complex Monge-Amp\`ere equation. Other parabolic flows over compact Hermitian manifolds were considered in \cite{LY1, TW3, TW4, TWY1}, where they obtained lots of interesting results parallel to those in K\"ahler case. By Tosatti-Weinkove's work, the author considers Donaldson's equation over compact Hermitian manifolds.

Let $(X,\omega)$ be a compact Hermitian manifold of the complex dimension $n$ and $\chi$ another Hermitian metric on $X$. We denote by $\mathcal{H}_{\chi}$ the set of all real-valued smooth functions $\varphi$ on $X$ such that $\chi_{\varphi}:=\chi+\sqrt{-1}
\partial\overline{\partial}\varphi>0$. Locally we have
\begin{equation}
\omega=\sqrt{-1}g_{i\bar{j}}dz^{i}\wedge dz^{\bar{j}}, \ \ \
\chi=\sqrt{-1}\chi_{i\bar{j}}dz^{i}\wedge dz^{\bar{j}}.\label{1.5}
\end{equation}
For any real positive $(1,1)$-form $\alpha:=\sqrt{-1}\alpha_{i\bar{j}}dz^{i}\wedge dz^{\bar{j}}$ and real $(1,1)$-form $\beta:=\sqrt{-1}\beta_{i\bar{j}}dz^{i}\wedge dz^{\bar{j}}$ we set
\begin{equation}
{\rm tr}_{\alpha}\beta:=\alpha^{i\bar{j}}\beta_{i\bar{j}}.\label{1.6}
\end{equation}
We consider Donaldson's equation
\begin{equation}
\omega\wedge\chi^{n-1}_{\varphi}=e^{F}\cdot \chi^{n}_{\varphi}, \ \ \ \varphi\in\mathcal{H}_{\varphi}\label{1.7}
\end{equation}
on $X$, where $F$ is a given smooth function on $X$.

The main result of this paper is the following a priori estimates.

\begin{theorem}\label{t1.1} Let $(X,\omega)$ be a compact Hermitian manifold of the complex dimension $n$ and $\chi$ another Hermitian metric. Let $\varphi$ be a smooth solution of Donaldson's equation (\ref{1.7}). Assume that
\begin{equation}
\chi-\frac{n-1}{ne^{F}}\omega>0.\label{1.8}
\end{equation}
Then
\begin{itemize}

\item[(1)] there exist uniform constant $A>0$ and $C>0$, depending only on $X, \omega,\chi$, and $F$, such that
    \begin{equation}
    {\rm tr}_{\omega}\chi_{\varphi}\leq C\cdot e^{A(\varphi-\inf_{X}\varphi)};
    \end{equation}

\item[(2)] there exists a uniform constant $C>0$, depending only on $X, \omega,\chi$, and $F$, such that
    \begin{equation}
    ||\varphi||_{C^{0}}\leq C;
    \end{equation}

\item[(3)] there are uniform $C^{\infty}$ a priori estimates on $\varphi$ depending only on $X,\omega,\chi$, and $F$.

\end{itemize}
\end{theorem}

Meanwhile, Guan, Li and Sun \cite{G1, GL2, GL3, GS1} considered a priori estimates for Donaldson's equation over compact Hermitian manifolds under very general structure conditions rather than the condition (\ref{1.8}).

\begin{remark}\label{r1.2} As remarked in \cite{TW1} (see page 22, line 27--28), to prove the zeroth estimate in Theorem \ref{t1.1} it suffices to show the second order estimate on $\varphi$. Our result gives an affirmative
answer to the question in \cite{TW1} (see page 22, line 28--30). Using the same argument in \cite{TW1} (page 33), we can get a $C^{\alpha}$ estimate on $\varphi$ for some $\alpha\in(0,1)$. Differentiating (\ref{1.7}) and applying the standard local elliptic estimates imply uniform $C^{\infty}$ estimates on $\varphi$.
\end{remark}

There are some natural questions about the equation (\ref{1.7}). Is condition (\ref{1.8})
sufficient to product a solution to (\ref{1.7})? When $\omega$ and $\chi$ both are K\"ahler, it has been proved in \cite{C2, W1, W2} that this condition is sufficient. The second question is to consider a
parabolic flow over compact Hermitian manifolds like the $J$-flow. Can we prove the long time existence and
convergence of such a flow? Song and Weinkove \cite{SW1} gave a
necessary and sufficient condition for existence of solutions to the
Donaldson's equation over compact K\"ahler manifolds (and also for convergence of the $J$-flow over compact K\"ahler manifolds). The last question then is whether we can find an analogous of above Song-Weinkove's condition. Those questions will be answered later.

\begin{remark}\label{r1.3} Here and henceforth, when we say a ``uniform constant'' it should be understood to be a constant that depends only on $X,\omega,\chi$, and $F$. We will often write $C$ or $C'$ for such a constant, where the value of $C$ or $C'$ may differ from line to line. For the relation $P\leq CQ$ for a uniform constant $C$ in the above sense, we write it as $P\lesssim Q$. ${\rm Re}(P)$ means the real part of $P$.
\end{remark}

{\bf Acknowledgement.} The author would like to thank Kefeng Liu, Valentino Tosatti, Xiaokui Yang for useful discussions on Donaldson's equation, the complex Monge-Amp\`ere equation and geometric flows. The author thanks referees's helpful suggestions.

\section{The second order estimates}

\subsection{Basic facts and notions}

Let $(X,\omega)$ be a complex Hermitian manifold of the complex dimension $n$ and $\chi$ another Hermitian metric on $X$. For a solution $\varphi$ of Donaldson's equation (\ref{1.7}), we denote by
\begin{equation}
\chi':=\chi+\sqrt{-1}\partial\overline{\partial}\varphi
=\sqrt{-1}(\chi_{i\bar{j}}+\varphi_{i\bar{j}})dz^{i}\wedge dz^{\bar{j}}.\label{2.1}
\end{equation}
Also, we set $\chi'_{i\bar{j}}:=\chi_{i\bar{j}}
+\varphi_{i\bar{j}}$. Then we observe that
\begin{equation}
{\rm tr}_{\chi'}\omega=n\frac{\omega\wedge(\chi')^{n-1}}{(\chi')^{n}}
=ne^{F}.\label{2.2}
\end{equation}
Consequently, ${\rm tr}_{\chi'}\omega$ is uniformly bounded away from zero and infinity. Let $\Delta_{\omega}$ denote the Laplacian operator of the Chern connection associated to the Hermitian metric $\omega$, and similarly for $\Delta_{\chi}$. Note that
\begin{equation}
{\rm tr}_{\omega}\chi'=g^{i\bar{j}}(\chi_{i\bar{j}}+\varphi_{i\bar{j}})
={\rm tr}_{\omega}\chi+\Delta_{\omega}\varphi.\label{2.3}
\end{equation}

\begin{remark}\label{r2.1} ${\rm tr}_{\omega}\chi'$ and ${\rm tr}_{\chi'}\omega$ are uniformly bounded from below away from zero. More precisely,
\begin{equation}
{\rm tr}_{\omega}\chi'\geq\frac{n}{e^{F}}, \ \ \ {\rm tr}_{\chi'}\omega=ne^{F}.
\end{equation}
The second assertion follows from (\ref{2.2}), while the first inequality is obtained as follows. We choose a normal coordinate system so that
\begin{equation*}
g_{i\bar{j}}=\delta_{ij}, \ \ \ \chi'_{i\bar{j}}=\lambda'_{i}\delta_{ij}
\end{equation*}
for some $\lambda'_{i},\cdots,\lambda'_{n}>0$. Donaldson's equation then yields
\begin{equation*}
ne^{F}=\sum_{1\leq i\leq n}\frac{1}{\lambda'_{i}}.
\end{equation*}
An elementary inequality shows that
\begin{equation*}
{\rm tr}_{\omega}\chi'=\sum_{1\leq i\leq n}\lambda'_{i}
\geq\frac{n^{2}}{\sum_{1\leq i\leq n}\frac{1}{\lambda'_{i}}}
=\frac{n^{2}}{ne^{F}}=\frac{n}{e^{F}}.
\end{equation*}
\end{remark}

We will frequently use the following

\begin{lemma}{\rm (Guan-Li \cite{GL1})}\label{l2.2} At any point $p\in X$ there exists a holomorphic coordinates system centered at $p$ such that, at $p$,
\begin{equation}
g_{i\bar{j}}=\delta_{ij}, \ \ \ \partial_{j}g_{i\bar{i}}=0\label{2.5}
\end{equation}
for all $i$ and $j$. Furthermore, we can assume that $\chi'_{i\bar{j}}$ is diagonal.
\end{lemma}

Let $\widetilde{\Delta}$ denote the Laplacian operator associated to the Hermitian metric $h_{i\bar{j}}$ whose inverse matrix is given by
\begin{equation}
h^{i\bar{j}}:=\chi'^{i\bar{\ell}}\chi'^{k\bar{j}}g_{k\bar{\ell}};\label{2.6}
\end{equation}
and $\widetilde{\nabla}$ the associated covariant derivatives.

The basic idea to obtain the second order estimate, following from the method of Yau \cite{Yau1}, is to consider the quantity
\begin{equation}
Q:=\log({\rm tr}_{\omega}\chi')-A\varphi\label{2.7}
\end{equation}
for some suitable constant $A$. Our first step is to estimate the term $\widetilde{\Delta}\log({\rm tr}_{\omega}\chi')$.

\begin{definition}\label{d2.3} For convenience, we say that a term $E$ is of {\bf type I} if
\begin{equation}
|E|_{\omega}\lesssim1,\label{2.8}
\end{equation}
and is of {\bf type II} if
\begin{equation}
|E|_{\omega}\lesssim{\rm tr}_{\omega}\chi'.\label{2.9}
\end{equation}
\end{definition}

It is east to see that any uniform constant is of type I and any type I term is of type II. We will use $E_{1}$ and $E_{2}$ to denote a type I and type II term, respectively.

\subsection{The estimate for $\widetilde{\Delta}\log({\rm tr}_{\omega}\chi')$}

Direct computation shows
\begin{equation}
\widetilde{\Delta}\log({\rm tr}_{\omega}\chi')
=\frac{\widetilde{\Delta}{\rm tr}_{\omega}\chi'}{{\rm tr}_{\omega}\chi'}
-\frac{|\widetilde{\nabla}{\rm tr}_{\omega}\chi'|^{2}_{h}}{({\rm tr}_{\omega}\chi')^{2}}.\label{2.10}
\end{equation}
By the definition, we have
\begin{eqnarray*}
\widetilde{\Delta}{\rm tr}_{\omega}\chi'&=&
h^{i\bar{j}}\partial_{i}\partial_{\bar{j}}(g^{k\bar{\ell}}
\chi'_{k\bar{\ell}})\\
&=&h^{i\bar{j}}\partial_{i}\left(-g^{k\bar{b}}g^{a\bar{\ell}}
\partial_{\bar{j}}g_{a\bar{b}}\cdot\chi'_{k\bar{\ell}}
+g^{k\bar{\ell}}\partial_{\bar{j}}\chi'_{k\bar{\ell}}\right)\\
&=&h^{i\bar{j}}
\left[g^{k\bar{\ell}}\partial_{i}\partial_{\bar{j}}
\chi'_{k\bar{\ell}}-g^{k\bar{b}}g^{a\bar{\ell}}
\partial_{i}g_{a\bar{b}}\cdot\partial_{\bar{j}}\chi'_{k\bar{\ell}}
-g^{k\bar{b}}g^{a\bar{\ell}}\partial_{\bar{j}}
g_{a\bar{b}}\cdot\partial_{i}\chi'_{k\bar{\ell}}\right.\\
&& \ -\left(-g^{k\bar{q}}g^{p\bar{b}}
\partial_{i}g_{p\bar{q}}\cdot g^{a\bar{\ell}}\partial_{\bar{j}}
g_{a\bar{b}}-g^{k\bar{b}}g^{a\bar{q}}g^{p\bar{\ell}}
\partial_{i}g_{p\bar{q}}\cdot\partial_{\bar{j}}g_{a\bar{b}}\right.\\
&& \ +\left.\left.g^{k\bar{b}}g^{a\bar{\ell}}\partial_{i}\partial_{\bar{j}}
g_{a\bar{b}}\right)\chi'_{k\bar{\ell}}\right].
\end{eqnarray*}
Using the local coordinates in Lemma \ref{l2.2}, we deduce that
\begin{eqnarray}
\widetilde{\Delta}{\rm tr}_{\omega}\chi'&=&\sum_{1\leq i,j\leq n}h^{i\bar{i}}
\partial_{i}\partial_{\bar{i}}\chi'_{k\bar{k}}-\sum_{1\leq i,k,\ell\leq n}
h^{i\bar{i}}\partial_{i}g_{\ell\bar{k}}\cdot\partial_{\bar{i}}
\chi'_{k\bar{\ell}}\nonumber\\
&&- \ \sum_{1\leq i,k,\ell\leq n}h^{i\bar{i}}\partial_{\bar{i}}
g_{\ell\bar{k}}\cdot\partial_{i}\chi'_{k\bar{\ell}}
+\sum_{1\leq i,k,p\leq n}h^{i\bar{i}}\partial_{i}g_{p\bar{k}}
\cdot\partial_{\bar{i}}g_{k\bar{p}}\cdot\chi'_{k\bar{k}}\nonumber\\
&&+ \ \sum_{1\leq i,k,q\leq n}h^{i\bar{i}}
\partial_{i}g_{k\bar{q}}\cdot\partial_{\overline{i}}g_{q\bar{k}}
\cdot\chi'_{k\bar{k}}-\sum_{1\leq i,k\leq n}
h^{i\bar{i}}\partial_{i}\partial_{\bar{i}}g_{k\bar{k}}
\cdot\chi'_{k\bar{k}}\label{2.11}\\
&=&\sum_{1\leq i,k\leq n}h^{i\bar{i}}\partial_{i}\partial_{\bar{i}}
\chi'_{k\bar{k}}-2\cdot{\rm Re}\left(\sum_{1\leq i,j,k\leq n}h^{i\bar{i}}
\partial_{\bar{i}}g_{j\bar{k}}\cdot\partial_{i}\chi'_{k\bar{j}}
\right)+E_{1},\nonumber
\end{eqnarray}
where
\begin{eqnarray*}
E_{1}&=&\sum_{1\leq i,j,k\leq n}h^{i\bar{i}}\partial_{i}g_{j\bar{k}}
\cdot\partial_{\bar{i}}g_{k\bar{j}}\cdot\chi'_{k\bar{k}}
+\sum_{1\leq i,j,k\leq n}h^{i\bar{i}}\partial_{i}g_{k\bar{j}}
\cdot\partial_{\bar{i}}g_{j\bar{k}}\cdot\chi'_{k\bar{k}}\\
&&- \ \sum_{1\leq i,k\leq n}h^{i\bar{i}}\partial_{i}\partial_{\bar{i}}
g_{k\bar{k}}\cdot\chi'_{k\bar{k}}.
\end{eqnarray*}
Since under the above mentioned local coordinates $\chi'_{i\bar{i}}
=\lambda'_{i}\delta_{ij}$, it follows that $h^{i\bar{i}}=
(\chi'^{i\bar{i}})^{2}=1/\lambda'^{2}_{i}$; hence $h^{i\bar{i}}\leq e^{2F}$ using Remark \ref{r2.1}. Therefore we
see that $E_{1}$ is of type II, i.e.,
\begin{equation}
|E_{1}|_{\omega}\lesssim{\rm tr}_{\omega}\chi'.\label{2.12}
\end{equation}

The first term on the right hand side of (\ref{2.11}) can be computed as
follows: From Donaldson's equation (\ref{1.7}), we obtain
\begin{equation*}
ne^{F}={\rm tr}_{\chi'}\omega=\chi'^{i\bar{j}}g_{i\bar{j}}
\end{equation*}
and, after taking the derivative with respect to $z^{\bar{\ell}}$,
\begin{equation*}
n\partial_{\bar{\ell}}F\cdot e^{F}=-\chi'^{i\bar{b}}
\chi'^{a\bar{j}}\partial_{\bar{\ell}}\chi'_{a\bar{b}}
\cdot g_{i\bar{j}}+\chi'^{i\bar{j}}\partial_{\bar{\ell}}g_{i\bar{j}}.
\end{equation*}
Differentiating above equation again with respect to $z^{k}$ yields
\begin{eqnarray*}
&&n\partial_{k}\partial_{\bar{\ell}}F\cdot e^{F}
+n\partial_{\bar{\ell}}F\partial_{k}F\cdot e^{F} \ \ = \ \ -\chi'^{i\bar{b}}\chi'^{a\bar{j}}
g_{i\bar{j}}\partial_{k}\partial_{\bar{\ell}}\chi'_{a\bar{b}}
-\chi'^{i\bar{b}}\chi'^{a\bar{j}}\partial_{\bar{\ell}}
\chi'_{a\bar{b}}\partial_{k}g_{i\bar{j}}\\
&& \ \ \ \ \ \ \ \ \ \ \ - \ \left(-\chi'^{i\bar{q}}\chi'^{p\bar{b}}
\partial_{k}\chi'_{p\bar{q}}\cdot\chi'^{a\bar{j}}g_{i\bar{j}}
-\chi'^{i\bar{b}}\chi'^{a\bar{q}}
\chi'^{p\bar{j}}\partial_{k}\chi'_{p\bar{q}}\cdot g_{i\bar{j}}\right)
\partial_{\bar{\ell}}\chi'_{a\bar{b}}\\
&& \ \ \ \ \ \ \ \ \ \ \ - \ \chi'^{i\bar{b}}\chi'^{a\bar{j}}
\partial_{k}\chi'_{a\bar{b}}\cdot\partial_{\bar{\ell}}g_{i\bar{j}}
+\chi'^{i\bar{j}}\partial_{k}\partial_{\bar{\ell}}g_{i\bar{j}}\\
&=&-\chi'^{i\bar{b}}\chi'^{a\bar{j}}g_{i\bar{j}}
\partial_{k}\partial_{\bar{\ell}}\chi'_{a\bar{b}}
-\chi'^{i\bar{b}}\chi'^{a\bar{j}}\partial_{k}\chi'_{a\bar{b}}
\cdot\partial_{\overline{\ell}}g_{i\bar{j}}+\chi'^{i\bar{j}}
\partial_{k}\partial_{\bar{\ell}}g_{i\bar{j}}\\
&&- \ \left(-\chi'^{i\bar{q}}\chi'^{p\bar{b}}
\partial_{k}\chi'_{p\bar{q}}\cdot\chi'^{a\bar{j}}g_{i\bar{j}}
-\chi'^{i\bar{b}}\chi'^{a\bar{q}}\partial_{k}\chi'_{p\bar{q}}
\cdot g_{i\bar{j}}+\chi'^{i\bar{b}}\chi'^{a\bar{j}}
\partial_{k}g_{i\bar{j}}\right)
\partial_{\bar{\ell}}\chi'_{a\bar{b}}.
\end{eqnarray*}
Multiplying above by $g^{k\bar{\ell}}$ on both sides implies
\begin{eqnarray*}
&&\left(\Delta_{\omega}F+|\nabla F|^{2}_{\omega}\right)ne^{F} \ \ = \ \ -\sum_{1\leq i,j,k,\ell\leq n}\left(h^{i\bar{j}}
g^{k\bar{\ell}}\partial_{k}\partial_{\bar{\ell}}\chi'_{i\bar{j}}-\chi'^{i\bar{j}}
g^{k\bar{\ell}}\partial_{k}\partial_{\bar{\ell}}g_{i\bar{j}}\right)\\
&&- \ \sum_{1\leq i,j,k,\ell,a,b\leq n}
\chi'^{i\bar{b}}\chi'^{a\bar{j}}g^{k\bar{\ell}}\partial_{k}
\chi'_{a\bar{b}}\cdot\partial_{\bar{\ell}}g_{i\bar{j}}
+\sum_{1\leq i,j,k,\ell, p, q\leq n}h^{i\bar{q}}\chi'^{p\bar{j}}
g^{k\bar{\ell}}\partial_{k}\chi'_{p\bar{q}}\cdot\partial_{\bar{\ell}}
\chi'_{i\bar{j}}\\
&&+\sum_{1\leq i,j,k,\ell,p,q\leq n}h^{p\bar{j}}
\chi'^{i\bar{q}}g^{k\bar{\ell}}\partial_{k}\chi'_{p\bar{q}}
\cdot\partial_{\bar{\ell}}\chi'_{i\bar{j}}
-\sum_{1\leq i,j,k,\ell,a,b\leq n}\chi'^{i\bar{b}}
\chi'^{a\bar{j}}g^{k\bar{\ell}}\partial_{\bar{k}}
g_{i\bar{j}}\cdot\partial_{\bar{\ell}}\chi'_{a\bar{b}}.
\end{eqnarray*}
Using the local coordinates (\ref{2.5}) we arrive at
\begin{eqnarray*}
&&\left(\Delta_{\omega}F+|\nabla F|^{2}_{\omega}\right)ne^{F}\\
&=&- \ \sum_{1\leq i,k\leq n}h^{i\bar{i}}\partial_{k}\partial_{\bar{k}}
\chi'_{i\bar{i}}+\sum_{1\leq i,k\leq n}\chi'^{i\bar{i}}\partial_{k}
\partial_{\bar{k}}g_{i\bar{i}}+\sum_{1\leq i,j,k\leq n}h^{i\bar{i}}
\chi'^{j\bar{j}}\partial_{k}\chi'_{j\bar{i}}\cdot\partial_{\bar{k}}
\chi'_{i\bar{j}}\\
&&+ \ \sum_{1\leq i,j,k\leq n}h^{i\bar{i}}\chi'^{j\bar{j}}
\partial_{k}
\chi'_{i\bar{j}}\cdot\partial_{\bar{k}}\chi'_{j\bar{i}}
-2\cdot{\rm Re}\left(\sum_{1\leq i,j,k\leq n}\chi'^{i\bar{i}}
\chi'^{j\bar{j}}\partial_{k}g_{i\bar{j}}\cdot\partial_{\bar{k}}
\chi'_{j\bar{i}}\right).
\end{eqnarray*}
Equivalently,
\begin{eqnarray}
\sum_{1\leq i,k\leq n}h^{i\bar{i}}\partial_{k}\partial_{\bar{k}}
\chi'_{i\bar{i}}&=&\sum_{1\leq i,j,k\leq n}h^{i\bar{i}}
\chi'^{j\bar{j}}\partial_{k}\chi'_{j\bar{i}}\partial_{\bar{k}}\chi'_{i\bar{j}}+
\sum_{1\leq i,j,k\leq n}h^{i\bar{i}}\chi'^{j\bar{j}}\partial_{k}
\chi'_{i\bar{j}}\partial_{\bar{k}}\chi'_{j\bar{i}}\nonumber\\
&&- \ 2\cdot{\rm Re}\left(\sum_{1\leq i,j,k\leq n}\chi'^{i\bar{i}}
\chi'^{j\bar{j}}\partial_{k}g_{i\bar{j}}\cdot\partial_{\bar{k}}
\chi'_{j\bar{i}}\right)\label{2.13}\\
&&+ \ \sum_{1\leq i,k\leq n}\chi'^{i\bar{i}}\partial_{k}
\partial_{\bar{k}}g_{i\bar{i}}-\left(\Delta_{\omega}F+|\nabla F|^{2}_{\omega}\right)
ne^{F}.\nonumber
\end{eqnarray}
Since
\begin{eqnarray*}
\partial_{k}\partial_{\bar{k}}\chi'_{i\bar{i}}&=&\partial_{k}\partial_{\bar{k}}
\left(\chi_{i\bar{i}}+\varphi_{i\bar{i}}\right)\\
&=&\partial_{k}
\partial_{\bar{k}}\chi_{i\bar{i}}+\partial_{k}
\partial_{\bar{k}}
\varphi_{i\bar{i}}\\
&=&\partial_{k}\partial_{\bar{k}}\chi_{i\bar{i}}
+\partial_{i}\partial_{\bar{i}}\varphi_{k\bar{k}}\\
&=&\partial_{k}\partial_{\bar{k}}\chi_{i\bar{i}}+\partial_{i}\partial_{\bar{i}}
\left(\chi'_{k\bar{k}}-\chi_{k\bar{k}}\right)\\
&=&\partial_{i}\partial_{\bar{i}}\chi'_{k\bar{k}}
+\left(\partial_{k}\partial_{\bar{k}}\chi_{i\bar{i}}-\partial_{i}
\partial_{\bar{i}}\chi_{k\bar{k}}\right),
\end{eqnarray*}
we conclude that
\begin{equation}
\sum_{1\leq i,k\leq n}h^{i\bar{i}}\partial_{i}\partial_{\bar{i}}
\chi'_{k\bar{k}}=\sum_{1\leq i,k\leq n}h^{i\bar{i}}\partial_{k}\partial_{\bar{k}}\chi'_{i\bar{i}}+\sum_{1\leq i,k\leq n}h^{i\bar{i}}
\left(\partial_{i}\partial_{\bar{i}}
\chi_{k\bar{k}}-\partial_{k}
\partial_{\bar{k}}\chi_{i\bar{i}}\right).\label{2.14}
\end{equation}
Combining (\ref{2.13}) and (\ref{2.14}) yields
\begin{eqnarray}
\sum_{1\leq i,k\leq n}h^{i\bar{i}}\partial_{i}\partial_{\bar{i}}
\chi'_{k\bar{k}}&=&\sum_{1\leq i,j,k\leq n}h^{i\bar{i}}
\chi'^{j\bar{j}}\partial_{k}\chi'_{j\bar{i}}\cdot\partial_{\bar{k}}
\chi'_{i\bar{j}}\nonumber\\
&&+ \ \sum_{1\leq i,j,k\leq n}h^{i\bar{i}}
\chi'^{j\bar{j}}\partial_{k}\chi'_{i\bar{j}}\partial_{\bar{k}}
\chi'_{j\bar{i}}\label{2.15}\\
&&- \ 2\cdot{\rm Re}\left(\sum_{1\leq i,j,k\leq n}
\chi'^{i\bar{i}}\chi'^{j\bar{j}}\partial_{k}g_{j\bar{i}}
\cdot\partial_{\bar{k}}\chi'_{i\bar{j}}\right)+E_{2},
\nonumber
\end{eqnarray}
where
\begin{equation*}
E_{2}=\sum_{1\leq i,k\leq n}\chi'^{i\bar{i}}
\partial_{k}\partial_{\bar{k}}g_{i\bar{i}}+\sum_{1\leq i,k\leq n}
h^{i\bar{i}}\left(\partial_{i}\partial_{\bar{i}}
\chi_{k\bar{k}}-\partial_{k}\partial_{\bar{k}}\chi_{i\bar{i}}\right)
-\left(\Delta_{\omega}F+|\nabla F|^{2}_{\omega}\right)ne^{F}.
\end{equation*}
By the same reason that $\chi'^{i\bar{i}}\leq e^{F}$ and $h^{i\bar{i}}\leq e^{2F}$, we observe
that $E_{2}$ is of type I and
\begin{equation}
|E_{2}|_{\omega}\lesssim1.\label{2.16}
\end{equation}
From (\ref{2.11}) and (\ref{2.15}), we get
\begin{eqnarray*}
\widetilde{\Delta}{\rm tr}_{\omega}\chi'&=&\sum_{1\leq i,j,k\leq n}
h^{i\bar{i}}\chi'^{j\bar{j}}\partial_{k}\chi'_{j\bar{i}}
\partial_{\bar{k}}\chi'_{i\bar{j}}+\sum_{1\leq i,j,k\leq n}
h^{i\bar{i}}\chi'^{j\bar{j}}\partial_{k}\chi'_{i\bar{j}}
\partial_{\bar{k}}\chi'_{j\bar{i}}\\
&&- \ 2\cdot{\rm Re}\left(\sum_{1\leq i,j,k\leq n}\chi'^{i\bar{i}}
\chi'^{j\bar{j}}\partial_{k}g_{j\bar{i}}\partial_{\bar{k}}
\chi'_{i\bar{j}}\right)\\
&&- \ 2\cdot{\rm Re}\left(\sum_{1\leq i,j,k\leq n}h^{i\bar{i}}
\partial_{\bar{i}}g_{j\bar{k}}\partial_{i}\chi'_{k\bar{j}}
\right)+E_{1}+E_{2}\\
&=&\sum_{1\leq i,j,k\leq n}h^{i\bar{i}}\chi'^{j\bar{j}}
\partial_{k}\chi'_{j\bar{i}}\partial_{\bar{k}}\chi'_{i\bar{j}}
+\sum_{1\leq i,j,k\leq n}h^{i\bar{i}}\chi'^{j\bar{j}}
\partial_{k}\chi'_{i\bar{j}}\partial_{\bar{k}}\chi'_{j\bar{i}}\\
&&- \  2\cdot{\rm Re}\left(\sum_{1\leq i,j,k\leq n}\chi'^{i\bar{i}}
\chi'^{j\bar{j}}\partial_{k}g_{j\bar{i}}\partial_{\bar{k}}
\chi'_{i\bar{j}}\right)\\
&&- \ 2\cdot{\rm Re}\left(\sum_{1\leq i,j,k\leq n}h^{i\bar{i}}
\partial_{\bar{i}}g_{j\bar{k}}\partial_{i}\chi'_{k\bar{j}}
\right)+E_{2},
\end{eqnarray*}
since any type I term is also of type II.

\subsection{The estimate for
$\widetilde{\Delta}{\rm log}({\rm tr}_{\omega}\chi')$, continued: $\omega$ is
K\"ahler}\label{subsection8.2.3}

In the case that $\omega$ is K\"ahler, we in addition have $\partial_{k}g_{ij}=0$ for any $i, j, k$ in Lemma \ref{l2.2}, and we deduce from the above equation that
\begin{equation}
\widetilde{\Delta}{\rm tr}_{\omega}\chi'=\sum_{1\leq i,j,k\leq n}
h^{i\bar{i}}\chi'^{j\bar{j}}\partial_{k}\chi'_{j\bar{i}}
\partial_{\bar{k}}\chi'_{i\bar{j}}
+\sum_{1\leq i,j,k\leq n}h^{i\bar{i}}\chi'^{j\bar{j}}
\partial_{k}\chi'_{i\bar{j}}\partial_{\bar{k}}\chi'_{j\bar{i}}
+E_{2}.\label{2.17}
\end{equation}
It remains to control the term $|\widetilde{\nabla}{\rm tr}_{\omega}\chi'|^{2}_{h}/({\rm tr}_{\omega}
\chi')^{2}$. Notice that
\begin{equation*}
\partial_{i}\left({\rm tr}_{\omega}\chi'\right)
=\partial_{i}\left(g^{k\bar{\ell}}\chi'_{k\bar{\ell}}\right)
=g^{k\bar{\ell}}\partial_{i}\chi'_{k\bar{\ell}}
=\sum_{1\leq k\leq n}\partial_{i}\chi'_{k\bar{k}}.
\end{equation*}
As in \cite{TW1}, we first give an inequality for $|\widetilde{\nabla}{\rm tr}_{\omega}\chi'|^{2}_{h}/{\rm tr}_{\omega}\chi'$ and then we control the term ${\rm Re}(\sum_{1\leq i,j\leq n}
h^{i\bar{i}}\chi'^{j\bar{j}}\partial_{i}\chi'_{j\bar{j}}
(\partial_{i}\chi_{j\bar{j}}-\partial_{\bar{j}}\chi_{j\bar{i}}))$. From
\begin{eqnarray*}
&&\frac{|\widetilde{\nabla}{\rm tr}_{\omega}\chi'|^{2}_{h}}{{\rm tr}_{\omega}\chi'} \ \
= \ \ \sum_{1\leq i,j,k\leq n}\frac{h^{i\bar{i}}
\partial_{i}\chi'_{j\bar{j}}\partial_{\bar{i}}
\chi'_{k\bar{k}}}{{\rm tr}_{\omega}\chi'} \ \ = \ \ \sum_{1\leq j,k,i\leq n}\frac{\sqrt{h^{i\bar{i}}}\partial_{i}\chi'_{j\bar{j}}
\sqrt{h^{i\bar{i}}}\partial_{\bar{i}}
\chi'_{k\bar{k}}}{{\rm tr}_{\omega}\chi'}\\
&\leq&\frac{1}{{\rm tr}_{\omega}\chi'}
\sum_{1\leq j,k\leq n}\left(\sum_{1\leq i\leq n}h^{i\bar{i}}
|\partial_{i}\chi'_{j\bar{j}}|^{2}\right)^{1/2}
\left(\sum_{1\leq i\leq n}h^{i\bar{i}}|\partial_{i}\chi'_{k\bar{k}}|^{2}\right)^{1/2}\\
&=&\frac{1}{{\rm tr}_{\omega}\chi'}
\left[\sum_{1\leq j\leq n}\left(\sum_{1\leq i\leq n}h^{i\bar{i}}
|\partial_{i}\chi'_{j\bar{j}}|^{2}\right)^{1/2}\right]^{2}\\
&=&\frac{1}{{\rm tr}_{\omega}\chi'}
\left[\sum_{1\leq j\leq n}\sqrt{\chi'_{j\bar{j}}}
\left(\sum_{1\leq i\leq n}h^{i\bar{i}}\chi'^{j\bar{j}}
|\partial_{i}\chi'_{j\bar{j}}|^{2}\right)^{1/2}\right]^{2}\\
&\leq&\sum_{1\leq i,j\leq n}h^{i\bar{i}}\chi'^{j\bar{j}}
|\partial_{i}\chi'_{j\bar{j}}|^{2} \ \ = \ \ \sum_{1\leq i,j\leq n}
h^{i\bar{i}}\chi'^{j\bar{j}}\partial_{i}\chi'_{j\bar{j}}
\partial_{\bar{i}}\chi'_{j\bar{j}}.
\end{eqnarray*}
From
\begin{eqnarray*}
\partial_{i}\chi'_{j\bar{j}}&=&\partial_{i}\left(\chi_{j\bar{j}}
+\varphi_{j\bar{j}}\right) \ \ = \ \ \partial_{i}\chi_{j\bar{j}}
+\partial_{j}\varphi_{i\bar{j}} \ \ = \ \ \partial_{i}\chi_{j\bar{j}}
-\partial_{j}\chi_{i\bar{j}}+\partial_{j}\chi'_{i\bar{j}},\\
\partial_{\bar{i}}\chi'_{j\bar{j}}&=&\partial_{\bar{i}}
\left(\chi_{j\bar{j}}+\varphi_{j\bar{j}}\right) \ \ = \ \
\partial_{\bar{i}}\chi_{j\bar{j}}+\partial_{\bar{j}}
\varphi_{j\bar{i}} \ \ = \ \ \partial_{\bar{i}}\chi_{j\bar{j}}
-\partial_{\bar{j}}\chi_{j\bar{i}}+\partial_{\bar{j}}
\chi'_{j\bar{i}},
\end{eqnarray*}
it follows that
\begin{eqnarray}
\frac{|\widetilde{\nabla}{\rm tr}_{\omega}\chi'|^{2}_{h}}{{\rm tr}_{\omega}\chi'}
&\leq&\sum_{1\leq i,j\leq n}h^{i\bar{i}}
\chi'^{j\bar{j}}\left(\partial_{j}\chi'_{i\bar{j}}
+\partial_{i}\chi_{j\bar{j}}-\partial_{j}\chi_{i\bar{j}}\right)
\left(\partial_{\bar{j}}\chi'_{j\bar{i}}+\partial_{\bar{i}}
\chi_{j\bar{j}}-\partial_{\bar{j}}\chi_{j\bar{i}}\right)\nonumber\\
&=&\sum_{1\leq i,j\leq n}h^{i\bar{i}}
\chi'^{j\bar{j}}\partial_{j}\chi'_{i\bar{j}}\partial_{\bar{j}}
\chi'_{j\bar{i}}+\sum_{1\leq i,j\leq n}
h^{i\bar{i}}\chi'^{j\bar{j}}\left|\partial_{i}\chi_{j\bar{j}}
-\partial_{j}\chi_{i\bar{j}}\right|^{2}\label{2.18}\\
&&+ \ 2\cdot{\rm Re}\left[\sum_{1\leq i,j\leq n}h^{i\bar{i}}\chi'^{j\bar{j}}
\partial_{j}\chi'_{i\bar{j}}\left(\partial_{\bar{i}}
\chi_{j\bar{j}}-\partial_{\bar{j}}\chi_{j\bar{i}}\right)
\right].\nonumber
\end{eqnarray}
Note that
\begin{equation}
\partial_{j}\chi'_{i\bar{j}}
=\partial_{j}\left(\chi_{i\bar{j}}
+\varphi_{i\bar{j}}\right)=\partial_{j}\chi_{i\bar{j}}
+\partial_{i}\varphi_{j\bar{j}}=\partial_{j}\chi_{i\bar{j}}
-\partial_{i}\chi_{j\bar{j}}+\partial_{i}\chi'_{j\bar{j}}.
\label{2.19}
\end{equation}
Substituting (\ref{2.19}) into (\ref{2.18}) we obtain
\begin{eqnarray}
&&\frac{|\widetilde{\nabla}{\rm tr}_{\omega}\chi'|^{2}_{h}}{{\rm tr}_{\omega}
\chi'} \ \ \leq \ \ \sum_{1\leq i,j\leq n}h^{i\bar{i}}\chi'^{j\bar{j}}
\partial_{j}\chi'_{i\bar{j}}\partial_{\bar{j}}\chi'_{j\bar{i}}
-\sum_{1\leq i,j\leq n}h^{i\bar{i}}\chi'^{j\bar{j}}
\left|\partial_{i}\chi_{j\bar{j}}-\partial_{j}\chi_{i\bar{j}}
\right|^{2}\nonumber\\
&& \ \ \ \ \ \ \ + \ 2\cdot{\rm Re}\left[\sum_{1\leq i,j\leq n}
h^{i\bar{i}}\chi'^{j\bar{j}}\partial_{i}\chi'_{j\bar{j}}
\left(\partial_{\bar{i}}\chi_{j\bar{j}}-\partial_{\bar{j}}
\chi_{j\bar{i}}\right)\right]\label{2.20}\\
&\leq&\sum_{1\leq i,j\leq n}h^{i\bar{i}}
\chi'^{j\bar{j}}\partial_{j}\chi'_{i\bar{j}}
\partial_{\bar{j}}\chi'_{j\bar{i}}+2\cdot{\rm Re}
\left[\sum_{1\leq i,j\leq n}h^{i\bar{i}}\chi'^{j\bar{j}}
\partial_{i}\chi'_{j\bar{j}}\left(\partial_{\bar{i}}
\chi_{j\bar{j}}-\partial_{\bar{j}}\chi_{j\bar{i}}\right)\right].\nonumber
\end{eqnarray}

\begin{lemma}\label{l2.4} If $\omega$ is K\"ahler, then $\widetilde{\Delta}{\rm log}
({\rm tr}_{\omega}\chi')\gtrsim-1$.
\end{lemma}

\begin{proof} Calculate, since $h^{j\bar{j}}=\chi'^{j\bar{j}}\chi'^{j\bar{j}}$,
\begin{eqnarray}
&&\left|2\cdot{\rm Re}\left[\sum_{1\leq i,j\leq n}h^{i\bar{i}}
\chi'^{j\bar{j}}\partial_{i}\chi'_{j\bar{j}}
\left(\partial_{i}\chi_{j\bar{j}}-\partial_{\bar{j}}\chi_{j\bar{i}}
\right)\right]\right|\nonumber\\
&=&\left|2\cdot{\rm Re}\left[\sum_{1\leq i,j\leq n}\sqrt{h^{j\bar{j}}}
\sqrt{\chi'^{j\bar{j}}}\partial_{i}\chi'_{j\bar{j}}
\cdot\sqrt{\chi'_{j\bar{j}}}h^{i\bar{i}}\left(\partial_{\bar{i}}
\chi_{j\bar{j}}-\partial_{\bar{j}}\chi_{j\bar{i}}\right)
\right]\right|\nonumber\\
&\leq&\sum_{1\leq i,j\leq n}h^{j\bar{j}}\chi'^{j\bar{j}}
\partial_{i}\chi'_{j\bar{j}}\partial_{\bar{i}}
\chi'_{j\bar{j}}+\sum_{1\leq i,j\leq n}\chi'_{j\bar{j}}(h^{i\bar{i}})^{2}
|\partial_{\bar{i}}\chi_{j\bar{j}}
-\partial_{\bar{j}}\chi_{j\bar{i}}|^{2}\label{2.21}\\
&\leq&\sum_{1\leq i,j,k\leq n}h^{k\bar{k}}\chi'^{j\bar{j}}
\partial_{i}\chi'_{j\bar{k}}\partial_{\bar{i}}\chi'_{k\bar{j}}
+E_{2}\nonumber\\
&=&\sum_{1\leq i,j,k\leq n}h^{i\bar{i}}
\chi'^{j\bar{j}}\partial_{k}\chi'_{j\bar{i}}
\partial_{\bar{k}}\chi'_{i\bar{j}}+E_{2},\nonumber
\end{eqnarray}
where $E_{2}$ is a term of type II:
\begin{equation*}
E_{2}=\sum_{1\leq i,j\leq n}\chi'_{j\bar{j}}(h^{i\bar{i}})^{2}
|\partial_{\bar{i}}\chi_{j\bar{j}}-\partial_{\bar{j}}
\chi_{j\bar{i}}|^{2}.
\end{equation*}
From (\ref{2.10}), (\ref{2.17}), (\ref{2.20}), and (\ref{2.21}), we have
\begin{eqnarray}
\widetilde{\Delta}{\rm log}({\rm tr}_{\omega}\chi')&\geq&\frac{1}{{\rm tr}_{\omega}\chi'}
\left[\sum_{1\leq i,j,k\leq n}h^{i\bar{i}}\chi'^{j\bar{j}}
\partial_{k}\chi'_{i\bar{j}}\partial_{\bar{k}}\chi'_{j\bar{i}}
+E_{2}\right.\nonumber\\
&& \ \left.-\sum_{1\leq i,j\leq n}h^{i\bar{i}}\chi'^{j\bar{j}}
\partial_{j}\chi'_{i\bar{j}}\partial_{\bar{j}}\chi'_{j\bar{i}}\right]\nonumber\\
&=&\frac{1}{{\rm tr}_{\omega}\chi'}\left(\sum_{1\leq i\leq n}\sum_{1\leq j\neq k\leq n}
h^{i\bar{i}}\chi'^{j\bar{j}}\partial_{k}\chi'_{i\bar{j}}
\partial_{\bar{k}}\chi'_{j\bar{i}}+E_{2}\right)\label{2.22}\\
&=&\frac{1}{{\rm tr}_{\omega}\chi'}\left(\sum_{1\leq i\leq n}\sum_{1\leq j\neq k\leq n}
h^{i\bar{i}}\chi'^{j\bar{j}}\left|\partial_{k}\chi'_{i\bar{j}}\right|^{2}
+E_{2}\right)\nonumber\\
&\geq&\frac{E_{2}}{{\rm tr}_{\omega}\chi'}.\nonumber
\end{eqnarray}
By the definition of type II terms, there exists a positive universal constant $C$ satisfying $|E_{2}|_{\omega}\leq C\cdot{\rm tr}_{\omega}\chi'$. Therefore
\begin{equation*}
\widetilde{\Delta}{\rm log}({\rm tr}_{\omega}\chi')\gtrsim-1.
\end{equation*}
Thus we complete the proof of the lemma.
\end{proof}

\begin{theorem}\label{t2.5} Let $(X,\omega)$ be a compact K\"ahler manifold of complex dimension $n$, and $\chi$ a Hermitian metric. Let $\varphi$ be a
smooth solution of Donaldson's equation
\begin{equation*}
\omega\wedge\chi^{n-1}_{\varphi}=e^{F}\chi^{n}_{\varphi}
\end{equation*}
where $F$ is a smooth function on $X$. Assume that
\begin{equation*}
\chi-\frac{n-1}{ne^{F}}\omega>0.
\end{equation*}
Then there are uniform constants $A>0$ and $C>0$, depending only on $X, \omega, \chi$, and $F$, such that
\begin{equation*}
{\rm tr}_{\omega}\chi_{\varphi}\leq C\cdot e^{A(\varphi-
\inf_{X}\varphi)}.
\end{equation*}
\end{theorem}

\begin{proof} Use the local coordinates in Lemma \ref{l2.2}. The proof is similar to that in \cite{W1, W2}. By the definition, one has
\begin{equation*}
\widetilde{\Delta}\varphi=h^{k\bar{k}}\varphi_{k\bar{k}}
=(\chi'^{k\bar{k}})^{2}(\chi'_{k\bar{k}}
-\chi_{k\bar{k}})=\sum_{1\leq k\leq n}\chi'^{k\bar{k}}
-{\rm tr}_{h}\chi={\rm tr}_{\chi'}\omega-{\rm tr}_{h}\chi.
\end{equation*}
Lemma \ref{l2.4} and (\ref{2.7}) imply that
\begin{eqnarray*}
\widetilde{\Delta}Q&=&\widetilde{\Delta}\left[{\rm log}
\left({\rm tr}_{\omega}\chi'\right)-A\varphi\right] \ \ \geq \ \ -C-
A\left({\rm tr}_{\chi'}\omega-{\rm tr}_{h}\chi\right)\\
&\geq&-C-A\sum_{1\leq i\leq n}\chi'^{i\bar{i}}+A\sum_{1\leq i\leq n}
\chi'^{i\bar{i}}\chi'^{i\bar{i}}\chi_{i\bar{i}}.
\end{eqnarray*}
Since $\varphi$ is a solution of Donaldson's equation, it follows that
${\rm tr}_{\chi'}\omega=ne^{F}$ by (\ref{2.7}) and hence, for any given positive uniform constants $A$ and $B$ (we will chose those constants later),
\begin{equation*}
\widetilde{\Delta}Q\geq\left(Bne^{F}-C\right)
-(A+B)\sum_{1\leq i\leq n}\chi'^{i\bar{i}}+A\sum_{1\leq i\leq n}\chi'^{i\bar{i}}
\chi'^{i\bar{i}}\chi_{i\bar{i}}.
\end{equation*}
By the assumption we have $\chi\geq\frac{n-1}{ne^{F}}(1+\epsilon)\omega$ for some suitable
number $\epsilon$ such that $0<\epsilon<\frac{1}{n-1}$. Let $p\in X$ be a point where $Q$ achieves its maximum; so $\widetilde{\Delta}Q\leq0$. At this point, we conclude that
\begin{eqnarray*}
0&\geq&\left(Bne^{F}-C\right)-(A+B)\sum_{1\leq i\leq n}\chi'^{i\bar{i}}
+A\sum_{1\leq i\leq n}\chi'^{i\bar{i}}\chi'^{i\bar{i}}\chi_{i\bar{i}}\\
&\geq&\left(Bne^{F}-C\right)-(A+B)\sum_{1\leq i\leq n}\chi'^{i\bar{i}}
+A\frac{n-1}{ne^{F}}(1+\epsilon)\sum_{1\leq i\leq n}\chi'^{i\bar{i}}\chi'^{i\bar{i}}.
\end{eqnarray*}
We denote by $\lambda'_{i}$ the eigenvalues of $\chi'$ at point $p$ such that
$\lambda'_{1}\leq\cdots\leq\lambda'_{n}$. Hence
\begin{equation*}
0\geq\left(Bne^{F}-C\right)-(A+B)\sum_{1\leq i\leq n}
\frac{1}{\lambda'_{i}}
+A\frac{n-1}{ne^{F}}(1+\epsilon)\sum_{1\leq i\leq n}\frac{1}{\lambda'^{2}_{i}}.
\end{equation*}
In order to obtain the upper bound for $\lambda'_{i}$ we need the following

\begin{lemma}\label{l2.6} Let $\lambda_{1},\cdots,\lambda_{n}$ be a
sequence of positive numbers. Suppose
\begin{equation*}
0\geq1-\alpha\sum_{1\leq i\leq n}\frac{1}{\lambda_{i}}+\beta\sum_{1\leq i\leq n}
\frac{1}{\lambda^{2}_{i}}\label{2.22}
\end{equation*}
for some $\alpha, \beta>0$ and $n\geq2$. If
\begin{equation}
\frac{4}{n}\leq\frac{\alpha^{2}}{\beta}<\frac{4}{n-1}\label{2.23}
\end{equation}
holds, then
\begin{equation}
\lambda_{i}\leq\frac{2\beta}{\alpha-\sqrt{n\alpha^{2}-4\beta}}
\label{2.24}
\end{equation}
for each $i$.
\end{lemma}

\begin{proof} Note that $\alpha-\sqrt{n\alpha^{2}-4\beta}>0$ by (\ref{2.23}). Since
\begin{equation*}
1+\sum_{1\leq i\leq n}\left(\frac{\alpha}{2\sqrt{\beta}}-\frac{\sqrt{\beta}}{\lambda_{i}}
\right)^{2}\leq\frac{n\alpha^{2}}{4\beta}
\end{equation*}
it implies that
\begin{equation*}
\sum_{1\leq i\leq n}\left(\frac{\alpha}{2\sqrt{\beta}}-\frac{\sqrt{\beta}}{\lambda_{i}}
\right)^{2}\leq\frac{n\alpha^{2}-4\beta}{4\beta}.
\end{equation*}
The right hand side of the above inequality is nonnegative by (\ref{2.23}). Consequently,
\begin{equation*}
\frac{\alpha}{2\sqrt{\beta}}-\frac{\sqrt{\beta}}{\lambda_{i}}
\leq\sqrt{\frac{n\alpha^{2}-4\beta}{4\beta}}
\end{equation*}
and then
\begin{equation*}
\frac{\alpha-\sqrt{n\alpha^{2}-4\beta}}{2\sqrt{\beta}}\leq
\frac{\sqrt{\beta}}{\lambda_{i}}.
\end{equation*}
Hence we obtain (\ref{2.24}).
\end{proof}

To apply Lemma \ref{l2.6}, we assume
\begin{equation}
Bne^{F}>C,\label{2.25}
\end{equation}
and set
\begin{equation}
\alpha\doteqdot\frac{A+B}{Bne^{F}-C}, \ \ \
\beta\doteqdot\frac{A\frac{n-1}{ne^{F}}(1+\epsilon)}{Bne^{F}-C}.
\label{2.26}
\end{equation}
In the following we will find the explicit formulas for $A$ and $B$ in terms
of $C$ such that the assumption (\ref{2.25}) and the condition
(\ref{2.23}) are both satisfied.

We choose a real number $\eta$ satisfying
\begin{equation}
0\leq\eta<1.\label{2.27}
\end{equation}
Set
\begin{equation}
\frac{\alpha^{2}}{\beta}=\frac{4}{n-\eta},\label{2.28}
\end{equation}
where $\alpha$ and $\beta$ are given in (\ref{2.26}). If (\ref{2.28}) was valid, then the condition (\ref{2.23}) is true. Equations (\ref{2.26}) and (\ref{2.28}) imply
\begin{equation*}
(A+B)^{2}=\frac{4}{n-\eta}(1+\epsilon)\left(Bne^{F}-C\right)
\frac{n-1}{ne^{F}}A
\end{equation*}
so that
\begin{equation*}
A^{2}+B^{2}+2\left[1-\frac{2(1+\epsilon)(n-1)}{n-\eta}\right]
AB+\frac{4(1+\epsilon)(n-1)C}{(n-\eta)ne^{F}}A=0.
\end{equation*}
The above relation can be rewritten as
\begin{eqnarray*}
\left[A+\left(1-\frac{2(1+\epsilon)(n-1)}{n-\eta}\right)
B\right]^{2}&=&\left[\left(1-\frac{2(1+\epsilon)(n-1)}{n-\eta}\right)^{2}
-1\right]B^{2}\\
&&- \ \frac{4(1+\epsilon)(n-1)C}{(n-\eta)ne^{F}}A.
\end{eqnarray*}
Taking
\begin{equation}
A=\left(-1+\frac{2(1+\epsilon)(n-1)}{n-\eta}\right)B\label{2.29}
\end{equation}
we have $A>B$ and
\begin{equation}
B=\frac{\frac{4(1+\epsilon)(n-1)C}{(n-\eta)ne^{F}}
\left(-1+\frac{2(1+\epsilon)(n-1)}{n-\eta}
\right)}{\left(1-\frac{2(1+\epsilon)(n-1)}{n-\eta}\right)^{2}-1}
=\frac{C}{ne^{F}}\cdot\frac{-(n-\eta)+2(1+\epsilon)(n-1)
}{-(n-\eta)
+(1+\epsilon)(n-1)},\label{2.30}
\end{equation}
assuming
\begin{equation}
(1+\epsilon)>\frac{n-\eta}{n-1}.\label{2.31}
\end{equation}
From (\ref{2.30}) and (\ref{2.31}) we see that
\begin{equation*}
\frac{Bne^{F}}{C}=\frac{-(n-\eta)+2(1+\epsilon)(n-1)
}{-(n-\eta)
+(1+\epsilon)(n-1)}>1.
\end{equation*}
From the assumption $0<\epsilon<\frac{1}{n-1}$ we have $0<n-(n-1)(1+\epsilon)<1$ and then such a $\eta$ always exists. Hence Lemma \ref{l2.6} yields
\begin{equation*}
\lambda'_{i}\leq\frac{2\beta}{\alpha-\sqrt{n\alpha^{2}-4\beta}}
\end{equation*}
where $\alpha$ and $\beta$ are determined by (\ref{2.26}), (\ref{2.29}), and (\ref{2.30}). Since ${\rm tr}_{\omega}\chi'=\sum^{n}_{i=1}\lambda'_{i}$, it follows that, at $p\in X$, ${\rm tr}_{\omega}\chi'\leq C$ for some uniform constant $C$ and, for any point $q\in X$,
\begin{equation*}
Q(q)\leq Q(p)={\rm log}({\rm tr}_{\omega}\chi')(p)-A\varphi(p)
\leq C-A\cdot\inf_{X}\varphi.
\end{equation*}
Equivalently, ${\rm log}({\rm tr}_{\omega}\chi')\leq C+A(\varphi
-\inf_{X}\varphi)$.
\end{proof}

\subsection{The estimate for $\widetilde{\Delta}{\rm log}({\rm tr}_{\omega}
\chi')$, continued: general case}\label{subsection8.2.4}

Now we consider the general case that both $\omega$ and $\chi$ may not be
K\"ahler. Using Lemma \ref{l2.2} we have
\begin{eqnarray}
\widetilde{\Delta}{\rm tr}_{\omega}\chi'&=&\sum_{1\leq i,j,k\leq n}h^{i\bar{i}}
\chi'^{j\bar{j}}\partial_{k}\chi'_{j\bar{i}}
\partial_{\bar{k}}\chi'_{i\bar{j}}+\sum_{1\leq i,j,k\leq n}h^{i\bar{i}}
\chi'^{j\bar{j}}\partial_{k}\chi'_{i\bar{j}}\partial_{\bar{k}}
\chi'_{j\bar{i}}+E_{2}\nonumber\\
&&- \ 2\cdot{\rm Re}\left(\sum_{1\leq i,j,k\leq n}\chi'^{i\bar{i}}
\chi'^{j\bar{j}}\partial_{k}g_{j\bar{i}}\partial_{\bar{k}}
\chi'_{i\bar{j}}\right)\label{2.32}\\
&&- \ 2\cdot{\rm Re}\left(\sum_{1\leq i,j,k\leq n}h^{i\bar{i}}
\partial_{\bar{i}}g_{j\bar{k}}\partial_{i}\chi'_{k\bar{j}}\right).\nonumber
\end{eqnarray}
As in \cite{TW1} we deal with the last two terms by using the local
coordinates in Lemma \ref{l2.2}. Starting from the last term, we calculate
\begin{eqnarray}
\sum_{1\leq i,j,k\leq n}h^{i\bar{i}}\partial_{i}\chi'_{k\bar{j}}
\partial_{\bar{i}}g_{j\bar{k}}&=&\sum_{1\leq i,j,k\leq n}h^{i\bar{i}}
\partial_{\bar{i}}g_{j\bar{k}}\partial_{i}\left(\chi_{k\bar{j}}
+\varphi_{k\bar{j}}\right)\nonumber\\
&=&\sum_{1\leq i,j,k\leq n}h^{i\bar{i}}\partial_{\bar{i}}
g_{j\bar{k}}\left(\partial_{i}\chi_{k\bar{j}}+\partial_{k}
\varphi_{i\bar{j}}\right)\nonumber\\
&=&\sum_{1\leq i,j,k\leq n}h^{i\bar{i}}\partial_{\bar{i}}
g_{j\bar{k}}\left(\partial_{i}\chi_{k\bar{j}}+\partial_{k}\chi'_{i\bar{j}}
-\partial_{k}\chi_{i\bar{j}}\right)\label{2.33}\\
&=&\sum^{n}_{i=1}\sum_{1\leq j\neq k\leq n}h^{i\bar{i}}
\partial_{\bar{i}}g_{j\bar{k}}\partial_{k}\chi'_{i\bar{j}}+E_{1},\nonumber
\end{eqnarray}
where $E_{1}$ is a term of type I and is given by
\begin{equation}
E_{1}=\sum_{1\leq i,j,k\leq n}h^{i\bar{i}}\partial_{\bar{i}}
g_{j\bar{k}}\left(\partial_{i}\chi_{k\bar{j}}-\partial_{k}\chi_{i\bar{j}}\right).
\label{2.34}
\end{equation}
Taking the real part of (\ref{2.33}) gives
\begin{eqnarray}
&& \ \ \ \ \ \ \ \ \ \ \left|2\cdot{\rm Re}\left(\sum_{1\leq i,j,k\leq n}h^{i\bar{i}}\partial_{i}
\chi'_{k\bar{j}}\partial_{\bar{i}}g_{j\bar{k}}\right)\right|
\label{2.35}\\
&=&\left|2\cdot{\rm Re}\left(\sum_{1\leq i\leq n}\sum_{1\leq j\neq k\leq n}
\sqrt{h^{i\bar{i}}}\sqrt{\chi'^{j\bar{j}}}\partial_{k}\chi'_{i\bar{j}}
\cdot\sqrt{h^{i\bar{i}}}\sqrt{\chi'_{j\bar{j}}}\partial_{\bar{i}}
g_{j\bar{k}}\right)\right|+E_{1}\nonumber\\
&\leq&\sum_{1\leq i\leq n}\sum_{1\leq j\neq k\leq n}h^{i\bar{i}}
\chi'^{j\bar{j}}\partial_{k}\chi'_{i\bar{j}}\partial_{\bar{k}}
\chi'_{j\bar{i}}+\sum^{n}_{i=1}\sum_{1\leq j\neq k\leq n}h^{i\bar{i}}
\chi'_{j\bar{j}}\partial_{\bar{i}}g_{j\bar{k}}\partial_{i}
g_{k\bar{j}}+E_{1}\nonumber\\
&\leq&\sum_{1\leq i\leq n}\sum_{1\leq j\neq k\leq n}h^{i\bar{i}}
\chi'^{j\bar{j}}\partial_{k}\chi'_{i\bar{j}}\partial_{\bar{k}}
\chi'_{j\bar{i}}+E_{2},\nonumber
\end{eqnarray}
since $\sum^{n}_{i=1}\sum_{1\leq j\neq k\leq n}h^{i\bar{i}}
\chi'_{j\bar{j}}\partial_{\bar{i}}g_{j\bar{k}}\partial_{i}
g_{k\bar{j}}$ is of type II. Similarly we have
\begin{eqnarray}
&& \ \ \ \ \ \ \ \ \ \ \left|2\cdot{\rm Re}\left(\sum_{1\leq i,j,k\leq n}\chi'^{i\bar{i}}
\chi'^{j\bar{j}}\partial_{\bar{k}}\chi'_{j\bar{i}}
\partial_{k}g_{i\bar{j}}\right)\right|\label{2.36}\\
&=&\left|2\cdot{\rm Re}\left(\sum_{1\leq i,j,k\leq n}
\sqrt{h^{j\bar{j}}}\sqrt{\chi'^{i\bar{i}}}\partial_{\bar{k}}
\chi'_{j\bar{i}}\cdot\sqrt{\chi'^{i\bar{i}}}\partial_{k}g_{i\bar{j}}\right)
\right|\nonumber\\
&\leq&\frac{1}{2}\sum_{1\leq i,j,k\leq n}h^{j\bar{j}}
\chi'^{i\bar{i}}\partial_{\bar{k}}\chi'_{j\bar{i}}
\partial_{k}\chi'_{i\bar{j}}+E_{1} \ \ = \ \ \frac{1}{2}
\sum_{1\leq i,j,k\leq n}h^{i\bar{i}}\chi'^{j\bar{j}}
\partial_{\bar{k}}\chi'_{i\bar{j}}\partial_{k}\chi'_{j\bar{i}}
+E_{1},\nonumber
\end{eqnarray}
where
\begin{equation}
E_{1}=2\sum_{1\leq i,j,k\leq n}\chi'^{i\bar{i}}\partial_{k}g_{i\bar{j}}
\partial_{\bar{k}}g_{j\bar{i}}\label{2.37}
\end{equation}
is a term of type I.

From (\ref{2.32}), (\ref{2.35}), and (\ref{2.36}), we conclude that
\begin{eqnarray}
\widetilde{\Delta}{\rm tr}_{\omega}\chi'&\geq&\sum_{1\leq i,j,k\leq n}h^{i\bar{i}}
\chi'^{j\bar{j}}\partial_{k}\chi'_{j\bar{i}}\partial_{\bar{k}}
\chi'_{i\bar{j}}+\sum_{1\leq i,j,k\leq n}h^{i\bar{i}}\chi'^{j\bar{j}}
\partial_{k}\chi'_{i\bar{j}}\partial_{\bar{k}}\chi'_{j\bar{i}}\nonumber\\
&&- \ \sum^{n}_{i=1}\sum_{1\leq j\neq k\leq n}h^{i\bar{i}}\chi'^{j\bar{j}}
\partial_{k}\chi'_{i\bar{j}}\partial_{\bar{k}}\chi'_{j\bar{i}}
-\frac{1}{2}\sum_{1\leq i,j,k\leq n}h^{i\bar{i}}
\chi'^{j\bar{j}}\partial_{k}\chi'_{j\bar{i}}
\partial_{\bar{k}}\chi'_{i\bar{j}}+E_{2}\label{2.38}\\
&=&\frac{1}{2}\sum_{1\leq i,j,k\leq n}h^{i\bar{i}}\chi'^{j\bar{j}}
\partial_{k}\chi'_{j\bar{i}}\partial_{\bar{k}}\chi'_{i\bar{j}}
+\sum_{1\leq i,j\leq n}h^{i\bar{i}}\chi'^{j\bar{j}}\partial_{j}\chi'_{i\bar{j}}
\partial_{\bar{j}}\chi'_{j\bar{i}}+E_{2}.\nonumber
\end{eqnarray}

It remains to control the term $|\widetilde{\nabla}{\rm tr}_{\omega}\chi'|^{2}_{h}/({\rm tr}_{\omega}\chi')^{2}$. As in (\ref{2.20}) one has
\begin{eqnarray}
\frac{|\widetilde{\nabla}{\rm tr}_{\omega}\chi'|^{2}_{h}}{{\rm tr}_{\omega}\chi'}
&\leq&\sum_{1\leq i,j\leq n}h^{i\bar{i}}\chi'^{j\bar{j}}
\partial_{j}\chi'_{i\bar{j}}\partial_{\bar{j}}\chi'_{j\bar{i}}
\nonumber\\
&&+ \ 2\cdot{\rm Re}\left[\sum_{1\leq i,j\leq n}h^{i\bar{i}}\chi'^{j\bar{j}}
\partial_{i}\chi'_{j\bar{j}}\left(\partial_{\bar{i}}\chi_{j\bar{j}}
-\partial_{\bar{j}}\chi_{j\bar{i}}\right)\right].\label{2.39}
\end{eqnarray}

\begin{lemma}\label{l2.7} One has $\widetilde{\Delta}{\rm log}({\rm tr}_{\omega}\chi')
\gtrsim-1$.
\end{lemma}

\begin{proof} As in the proof of Lemma \ref{l2.4} we have
\begin{eqnarray}
&& \ \ \ \ \ \ \ \ \ \ \left|2\cdot{\rm Re}\left[\sum_{1\leq i,j\leq n}h^{i\bar{i}}
\chi'^{i\bar{i}}\partial_{i}\chi'_{j\bar{j}}
\left(\partial_{\bar{i}}\chi_{j\bar{j}}-\partial_{\bar{j}}
\chi_{j\bar{i}}\right)\right]\right|\label{2.40}\\
&=&\left|2\cdot{\rm Re}\left[\sum_{1\leq i,j\leq n}
\sqrt{h^{j\bar{j}}}\sqrt{\chi'^{j\bar{j}}}
\partial_{i}\chi'_{j\bar{j}}\cdot\sqrt{\chi'_{j\bar{j}}}
h^{i\bar{i}}\left(\partial_{i}\chi_{j\bar{j}}
-\partial_{\bar{j}}\chi_{j\bar{i}}\right)\right]\right|\nonumber\\
&\leq&\frac{1}{2}\sum_{1\leq i,j\leq n}h^{j\bar{j}}
\chi'^{j\bar{j}}\partial_{i}\chi'_{j\bar{j}}
\partial_{\bar{i}}\chi'_{j\bar{j}}+2\sum_{1\leq i,j\leq n}
\chi'_{j\bar{j}}(h^{i\bar{i}})^{2}
\left|\partial_{\bar{i}}\chi_{j\bar{j}}-\partial_{\bar{j}}
\chi_{j\bar{i}}\right|^{2}\nonumber\\
&\leq&\frac{1}{2}\sum_{1\leq i,j,k\leq n}h^{k\bar{k}}
\chi'^{j\bar{j}}\partial_{i}\chi'_{j\bar{k}}
\partial_{\bar{i}}\chi'_{k\bar{j}}+E_{2} \ \ = \ \
\frac{1}{2}\sum_{1\leq i,j,k\leq n}h^{i\bar{i}}
\chi'^{j\bar{j}}\partial_{k}\chi'_{j\bar{i}}
\partial_{\bar{k}}\chi'_{i\bar{j}}+E_{2},\nonumber
\end{eqnarray}
where $E_{2}$ is a term of type II and given by
\begin{equation*}
E_{2}=2\sum_{1\leq i,j\leq n}\chi'_{j\bar{j}}
(h^{i\bar{i}})^{2}\left|\partial_{\bar{i}}
\chi_{j\bar{j}}-\partial_{\bar{j}}\chi_{j\bar{i}}\right|^{2}.
\end{equation*}
Combining (\ref{2.40}) with (\ref{2.10}), (\ref{2.38}), and
(\ref{2.39}), we arrive at
\begin{eqnarray*}
&&\widetilde{\Delta}{\rm log}({\rm tr}_{\omega}\chi')\geq\frac{1}{{\rm tr}_{\omega}
\chi'}\left[\frac{1}{2}\sum_{1\leq i,j,k\leq n}h^{i\bar{i}}\chi'^{j\bar{j}}
\partial_{k}\chi'_{j\bar{i}}\partial_{\bar{k}}\chi'_{i\bar{j}}
+\sum_{1\leq i,j\leq n}h^{i\bar{i}}\chi'^{j\bar{j}}
\partial_{j}\chi'_{i\bar{j}}\partial_{\bar{j}}\chi'_{j\bar{i}}\right.\\
&&-\left.\sum_{1\leq i,j\leq n}h^{i\bar{i}}\chi'^{j\bar{j}}
\partial_{j}\chi'_{i\bar{j}}\partial_{\bar{j}}\chi'_{j\bar{i}}
-\frac{1}{2}\sum_{1\leq i,j,k\leq n}h^{i\bar{i}}\chi'^{j\bar{j}}
\partial_{k}\chi'_{j\bar{i}}\partial_{\bar{k}}
\chi'_{i\bar{j}}+E_{2}\right] \ \ = \ \ \frac{E_{2}}{{\rm tr}_{\omega}\chi'}.
\end{eqnarray*}
By the definition of type II terms, there exists a positive uniform constant $C$ satisfying $|E_{2}|_{\omega}\leq C\cdot{\rm tr}_{\omega}\chi'$. Therefore
\begin{equation*}
\widetilde{\Delta}{\rm log}({\rm tr}_{\omega}\chi')\geq-C.
\end{equation*}
This complete the proof.
\end{proof}

By using the similar method as in the proof of Theorem \ref{t2.5}, we have

\begin{theorem}\label{t2.8} Let $(X, \omega)$ be a compact Hermitian manifold of the complex dimension $n$, and $\chi$ another Hermitian metric. Let $\varphi$ be a smooth solution of Donaldson's equation
\begin{equation*}
\omega\wedge\chi^{n-1}_{\varphi}=e^{F}\chi^{n}_{\varphi},
\end{equation*}
where $F$ is a smooth function on $X$. Assume that
\begin{equation*}
\chi-\frac{n-1}{ne^{F}}\omega>0.
\end{equation*}
Then there are uniform constants $A>0$ and $C>0$, depending only on $X, \omega, \chi$, and $F$, such that
\begin{equation*}
{\rm tr}_{\omega}\chi_{\varphi}\leq C\cdot e^{A(\varphi-\inf_{X}\varphi)}.
\end{equation*}
\end{theorem}




\begin{thebibliography}{10}


\bibitem{C1} Chen, Xiuxiong. \textit{On the lower bound of the Mabuchi energy and its application}, Internat. Math. Res. Notices, 2000, no. 12, 607--623. MR 1772078 (2001 f: 32042)

\bibitem{C2} Chen, Xiuxiong. \textit{A new parabolic flow in K\"ahler manifolds}, Commm. Anal. Geom., {\bf 12}(2004), no. 4, 837--852. MR
    2104078 (2005 h: 53116)

\bibitem{D1} Donaldson, S. K. \textit{Moment maps and diffeomorphisms}, Asian J. Math., {\bf 3}(1999), no. 1, 1--15. MR 1701920 (2001 a: 53122)


\bibitem{FL1} Fang, Hao; Lai, Mijia. \textit{On the geometric flows
solving K\"ahlerian inverse $\sigma_{k}$ equations}, Pacific J. Math., {\bf 258}(2012), no. 2, 291--304. MR 2981955

\bibitem{FL2} Fang, Hao; Lai, Mijia. \textit{Convergence of general
inverse $\sigma_{k}$-flow on K\"ahler manifolds with Calabi Ansatz}, arXiv: 1203.5253. (To appear in Transactions of the American Mathematical
Society)

\bibitem{FLM} Fang, Hao; Lai, Mijia; Ma, Xinan. \textit{On a class
of fully nonlinear flow in K\"ahler geometry}, J. Reine Angew. Math., {\bf 653}(2011), 189--220. MR 2794631 (2012 g: 53134)

\bibitem{FLSW} Fang, Hao; Lai, Mijia; Song, Jian; Weinkove, Ben. \textit{The
$J$-flow on K\"ahler surfaces: a boundary case}, arXiv: 1204.4068.


\bibitem{Gill1} Gill, Matt. \textit{Convergence of the parabolic complex Monge-Amp\`ere equation on compact Hermitian manifolds}, Comm. Anal. Geom., {\bf 19}(2011), no. 2, 277--303. MR 2835881 (2012 h: 32047)


\bibitem{G1} Guan, Bo. \textit{Second order estimates and regularity for
fully nonlinear elliptic equations on Riemannian manifolds}, arxiv: 1211.0181.

\bibitem{GL1} Guan Bo; Li Qun. \textit{Complex Monge-Amp\`ere equations and totally real submanifolds}, Adv. Math., {\bf 225}(2010), no. 3, 1185--1223. MR 2673728 (2011 g: 32053)

\bibitem{GL2} Guan, Bo; Li, Qun. \textit{A Monge-Amp\`ere type fully nonlinear equation on Hermitian manifolds}, Discrete Contin. Dyn. Syst. Ser. B, {\bf 17}(2012), no. 6, 1991--1999. MR 2924449

\bibitem{GL3} Guan, Bo; Li, Qun. \textit{The Dirichlet problem for a complex Monge-Amp\`ere type equation on Hermitian manifolds}, arXiv: 1210.5526.


\bibitem{GS1} Guan, Bo; Sun, Wei. \textit{On a class of fully nonlinear
elliptic equations on Hermitian manifolds}, arXiv: 1301.5863.


\bibitem{LY1} Liu, Ke-Feng; Yang, Xiao-Kui. \textit{Geometry of Hermitian
manifolds}, Internat. J. Math., {\bf 23}(2012), no. 6, 1250055, 40pp. MR 2925476.


\bibitem{SW} Song, Jian; Weinkove, Ben. \textit{On Donaldson's flow of surfaces in a hyperk\"ahler four-manifold}, Math. Z., {\bf 256}(2007), no. 4, 769--787. MR 2308890 (2008 b: 53090)


\bibitem{SW1} Song, Jian; Weinkove, Ben. \textit{On the convergence and singularities of the $J$-flow with applications to the Mabuchi energy}, Comm. Pure Appl. Math., {\bf 61}(2008), no. 2, 210--229. MR 2368374 (2009 a: 32038)

\bibitem{TW1} Tosatti, Valentino; Weinkove, Ben. \textit{Estimates for the complex Monge-Amp\`ere equation on Hermitian and balanced manifolds}, Asian J. Math., {\bf 14}(2010), no. 1, 19--40. MR 2726593 (2011 h:
    32043)

\bibitem{TW2} Tosatti, Valentino; Weinkove, Ben. \textit{The complex Monge-Amp\`ere equation on compact Hermitian manifolds}, J. Amer. Math. Soc., {\bf 23}(2010), no. 4, 1187--1195. MR 2669712 (2012 c: 32055)



\bibitem{TW3} Tosatti, Valentino; Weinkove, Ben. \textit{On the evolution of
a Hermitian metric by its Chern-Ricci form}, arxiv: 1201.0312v2.


\bibitem{TW4} Tosatti, Valentino; Weinkove, Ben. \textit{The chern-Ricci
flow on complex surfaces}, arXiv:1209.2663.

\bibitem{TWY1} Tosatti, Valentino; Weinkove, Ben; Yang, Xiaokui. \textit{Collapsing of the Chern-Ricci flow on elliptic surfaces}, arXiv:
    1302.6545.



\bibitem{W1} Weinkove, Ben. \textit{Convergence of the $J$-flow on K\"ahler surfaces}, Comm. Anal. Geom., {\bf 12}(2004), no. 4, 949--965. MR 2104082 (2005 g: 32027)

\bibitem{W2} Weinkove, Ben. \textit{On the $J$-flow in higher dimensions and the lower boundedness of the Mabuchi energy}, J. Differential. Geom., {\bf 73}(2006), no. 2, 451--358. MR 2226957 (2007 a: 32026)

\bibitem{Yau1} Yau, Shing Tung. \textit{On the Ricci curvature of a compact K\"ahler manifold and the complex Monge-Amp\`ere equation, I}, Comm. Pure. Appl. Math., {\bf 31}(1978), no. 3, 339--411. MR 0480350 (81 d: 53045)

\bibitem{Z1} Zhang, Xiangwen. \textit{A priori estimates for complex
Monge-Amp\`ere equation on Hermitian manifolds}, Int. Math. Res. Not. IMRN 2010, no. 19, 3814--3836. MR 2725515 (2011 k: 32057)

\bibitem{ZZ1} Zhang, Xi; Zhang, Xiangwen. \textit{Regularity estimates of
solutions of complex Monge-Amp\`ere equations on Hermitian manifolds}, J. Funct. Anal., {\bf 260}(2011), no. 7, 2004--2026.

\end{thebibliography}
\end{document}